\documentclass[12pt]{article}

\usepackage[top=2cm, bottom=2cm, left=2cm, right=2cm]{geometry}
\usepackage[usenames]{color}
\usepackage{url}
\usepackage{subfigure}
\usepackage{amsfonts,mathrsfs,color}
\usepackage{amssymb,amsmath,amsthm}
\usepackage{verbatim}
\usepackage{acronym}
\usepackage{graphicx}
\usepackage{enumerate}
\newcommand{\remove}[1]{}

\def\fskip#1{}

\newtheorem{theorem}{Theorem}

\newtheorem{corollary}{Corollary}

\newtheorem{definition}{Definition}

\newtheorem{lemma}{Lemma}

\newcommand{\Ac}{\{A(k)\}}

\newcommand{\Bc}{\{B(k)\}}

\newcommand{\bs}{\bar{S}}

\newcommand{\diag}{\mbox{diag}}

\newcommand{\Einf}{\mathcal{E}^\infty}

\newcommand{\Fc}{\{\mathcal{F}_k\}}

\newcommand{\Ginf}{G^{\infty}}

\newcommand{\pc}{\{\pi(k)\}}

\newcommand{\Pstar}{\mathcal{P}^*}

\newcommand{\Rm}{\mathbb{R}^m}
\newcommand{\Rmo}{\Z^+\times \mathbb{R}^{m}}

\newcommand{\Uc}{\{U(k)\}}

\newcommand{\Wc}{\{W(k)\}}

\newcommand{\xc}{\{x(k)\}}

\def\1{{\bf 1}}

\def\F{\mathcal{F}}

\newcommand{\al}{\alpha}

\def\R{\mathbb{R}}
\def\re{\mathbb{R}}

\def\S{\mathscr{S}}

\def\Z{\mathbb{Z}}
\newcommand{\EXP}[1]{\mathsf{E}\!\left[#1\right] }

\begin{document}
\title{Product of Random Stochastic Matrices}
\author{Behrouz Touri and Angelia Nedi\'c\thanks{Coordinated Science Laboratory, University of
Illinois, Urbana, IL 61801,
Email: \{touri1,angelia\} @illinois.edu. This
research was supported by the Air Force Office of Scientific Research
(AFOSR) Complex Networks Program under grant number FA9550-09-1-05205 and
the National Science Foundation under CAREER grant CMMI
07-42538}}
\date{}
\maketitle
\begin{abstract}
The paper deals with the convergence properties of the products of random (row-)stochastic matrices.
The limiting behavior of such products is studied from a dynamical system point of view.
In particular, by appropriately defining a dynamic associated with a given
sequence of random (row-)stochastic matrices, we prove that the dynamics
admits a class of time-varying Lyapunov functions, including a quadratic one.
Then, we discuss a special class of stochastic matrices, a class $\Pstar$,
which plays a central role in this work. We then introduce balanced chains and using some geometric properties of these chains, we characterize the stability of a subclass of balanced chains.
As a special consequence of this stability result,
we obtain an extension of a central result in the non-negative matrix theory
stating that, for any aperiodic and irreducible row-stochastic matrix $A$, the limit
$\lim_{k\to\infty}A^k$ exists and
it is a rank one stochastic matrix. We show that a generalization of this result holds not only for
sequences of stochastic matrices but also for
independent random sequences of such matrices.
\end{abstract}

\section{Introduction}
Averaging dynamics or distributed averaging dynamics has played a fundamental role in the recent
studies of various distributed systems and algorithms.
Examples of such distributed problems and algorithms include distributed optimization \cite{Tsitsiklis86,Nedic09,Nedic_cdc07,Johansson08a}, distributed control of robotic networks \cite{DistCtrlRobotNetw}, and study of opinion dynamics in social networks \cite{Krause:97,Krause:02}.

The study of averaging dynamics is closely related to the study of products of stochastic matrices.
Such products have been studied from two perspectives: the theory of Markov chains and the distributed averaging settings. The notable works in the domain of the theory of Markov chain are the early studies of Hajnal and Wolfowitz in \cite{hajnal} and \cite{wolf}, respectively, where sufficient conditions are derived for
the convergence of the products of row-stochastic matrices to a rank one matrix.
The exploration of this domain from the distributed averaging perspective was started by the work of \cite{SenetaCons} and the seminal work of J.~Tsitsiklis \cite{Tsitsiklis84}.

Products of \textit{random} stochastic matrices have also attracted many mathematicians
as such products are examples of convolutions of probability measures on semigroups \cite{mukherjea1976measures,rosenblatt1964products,grenander1963probabilities,rosenblatt1960limits}. Due to the technicalities involved, such studies are confined to identically independently distributed (i.i.d.) random chains or their generalizations in the domain of stationary ergodic chains. From the engineering perspective, this area have been recently explored in
\cite{Tahbaz-Salehi08,Tahbaz-Salehi09,ErgodicityPaper,TouriNedicCDC2011,
TouriNedicCDC20112,TouriNedich:Approx,TouriNedich:Doubly}.
In \cite{Tahbaz-Salehi08}, a necessary and sufficient condition for ergodicity of products of i.i.d.\ stochastic matrices has been derived. A generalization of such result for stationary and ergodic chains is discussed \cite{Tahbaz-Salehi09}. In a sequence of papers, \cite{ErgodicityPaper,TouriNedich:Approx,TouriNedicCDC2011,TouriNedicCDC20112}, we have defined fundamental concepts of \textit{infinite flow property}, \textit{infinite flow graph}, and \textit{$\ell_1$-approximation}. We have showed that these properties are very closely related to
the convergence properties of the product of random stochastic matrices that are not necessarily identically distributed. The current work is a continuation of the line of the aforementioned papers.

In particular, in this paper, we derive a set of necessary and sufficient conditions for ergodicity and convergence of a product of independent random stochastic matrices.
Specifically, we study a class of random stochastic matrices, which we refer to as \textit{balanced chains} and show that this class contains many of the previously studied chains of random and deterministic stochastic matrices. This property was first introduced in our earlier work \cite{TouriNedicCDC2011} for discrete-time dynamics and in \cite{BalancedHenderickx} for continuous-time dynamics. Much research has been done on such a criterion since then (see e.g.\ \cite{shi2011towards,martin2011continuous,bolouki2012theorems}).
Unlike the prior work, our work adopts dynamical system point of view for the averaging dynamics:
we first draw the connection between the products of random stochastic matrices and
the random dynamics driven by such matrices. Then, we show that every dynamics driven by a chain of independent random stochastic matrices admits
a \textbf{time-varying quadratic Lyapunov function}. In fact, we show more by establishing
that for any convex function, there exists a Lyapunov function adjusted to such a convex function.
This result opens up a new window for the study of averaging dynamics and distributed algorithms, as quadratic Lyapunov function has proven to be a powerful tool to study such dynamics for different sub-classes of stochastic chains (see e.g., \cite{alex_ange08} and \cite{ErgodicityPaper}). However, the non-existence of quadratic \textbf{time-invariant} Lyapunov functions was suspected for general class of averaging dynamics \cite{Jadbabaie03}, and it was proven later in \cite{Alex:Nonexistence}.

After proving the existence of time-varying quadratic Lyapunov functions for averaging dynamics,
we introduce a special class of stochastic chains, $\Pstar$ chains, and we show that the products of matrices
drawn from this sub-class converge almost surely. We then provide the definition of balanced-ness for random stochastic chains and we show that many previously studied classes of stochastic chains are
examples of such balanced chains. Finally, using a geometric property of the balanced chains,
we show that such chains are examples of $\Pstar$ chains, which leads to the main result of this paper which
can be interpreted as a generalization of the known convergence of $A^k$ to a rank-one stochastic matrix (for
aperiodic and irreducible stochastic matrix $A$) to the case of
inhomogeneous chains of stochastic matrices, as well as independent random chains of such matrices.

The contribution of this work is as follows: 1) we prove the existence of a family of time-varying Lyapunov functions for random averaging dynamics; 2) we introduce class $\Pstar$ of random stochastic chains, and provide necessary and sufficient conditions for the stability of the corresponding dynamics;
3) we introduce balanced chains and we use them to establish some necessary and sufficient conditions
for the stability of the resulting dynamics; and 4) we provide an extension of the fundamental
convergence result in the non-negative matrix theory.

The paper is organized as follows: in Section \ref{sec:setting}, we formulate and provide the random setting for our study of the products of random stochastic matrices which will be considered throughout the paper.
In Section \ref{sec:dynamics}, we study the dynamics related to such matrices and draw the connection between the study of such products and the associated dynamics, and we prove that the dynamics
admits a class of time-varying Lyapunov functions, including a quadratic one.
Then, we discuss the class $\Pstar$ in Section~\ref{sec:infflowstable:Pstar} which plays a central role in our development. We then introduce the class of balanced chains and using the geometric structure of
these chains, as well as the developed results in the preceding sections,
we characterize the stability of a subclass of those chains.
Finally, in Section~\ref{sec:connection}, we apply the developed results to prove an extension of a central result in the non-negative matrix theory on the convergence of $A^k$ to a rank-one matrix.
We conclude this work by a discussion in Section~\ref{sec:conclusion}.

\section{Problem Setting}\label{sec:setting}
We work exclusively with row-stochastic matrices, so we simply refer to them as stochastic matrices.
Let $(\Omega,\F,\Pr{})$ be a probability space and
let $\Wc$ be a chain of $m\times m$ random stochastic matrices,
i.e.\ for all $k\geq 1$, the matrix  $W(k)$ is a stochastic almost surely and
$W_{ij}(k):\Omega\to\R$ is a Borel-measurable function for all $i,j\in [m]$,
where $[m]=\{1,\ldots,m\}$.
Throughout this paper, we denote random sequences of stochastic matrices by last alphabet letters such as
$\Wc$ and $\Uc$, and we use the first alphabet letters such as $\Ac$ and $\Bc$
to denote deterministic sequences of stochastic matrices. We also refer to a sequence of stochastic matrices
as a {\it stochastic chain}, or just simply as a chain.

 Let $\Wc$ be an independent random chain. Then, we say that
$\Wc$ is \textit{strongly aperiodic} if there exists a $\gamma\in(0,1]$
such that
 \[\EXP{W_{ii}(k)W_{ij}(k)}\geq \gamma \EXP{W_{ij}(k)}\qquad
\hbox{for all $i\not=j\in[m]$ and all $k\geq 1$}. \]
Note that if $W_{ii}(k)\geq \gamma$ almost surely for all $i\in[m]$ and all
$k\geq 1$, then such a chain is strongly aperiodic. Also, note that
by summing both sides of the above inequality over $j\not=i$, we obtain
 \[\EXP{W_{ii}(k)}\geq\EXP{W_{ii}(k)(1-W_{ii}(k))}\geq
\gamma (1-\EXP{W_{ii}(k)}).\]
Hence, $\EXP{W_{ii}(k)}\geq \frac{\gamma}{1-\gamma}$ for all $i\in[m]$ and
all $k\geq 1$. Thus, for a strongly aperiodic chain $\Wc$, the expected chain
$\{\EXP{W(k)}\}$ is strongly aperiodic. It follows that a
deterministic chain $\Ac$ is strongly aperiodic if and only if
$A_{ii}(k)\geq \tilde\gamma$ for some $\tilde \gamma>0$, and
for all $i\in[m]$ and $k\geq 1$.

For the subsequent use, for an $m\times m$ random (or deterministic) matrix
$W$ and a non-trivial index set
$S\subset [m]$ (i.e. $S\not=\emptyset$ and $S\not=[m]$),
we define  the quantity $W_{S\bar{S}}=\sum_{i\in S,j\in\bar{S}}W_{ij}$,
where $\bar S$ is the complement of the index set $S$.

We say that an independent random chain $\Wc$ is \textit{balanced}
if there exists some $\al>0$ such that
 \begin{align}\label{eqn:balanceddef}
    \EXP{W_{S\bar{S}}(k)}\geq \al\EXP{W_{\bar{S}S}(k)}\qquad\mbox{
for all nontrivial $S\subset [m]$ and all $k\geq 1$}.
  \end{align}
From this definition, it can be seen that $\al\leq 1$.

Finally, with a given random chain $\Wc$, let us associate a random graph
$\Ginf=([m],\Einf)$ with the vertex set $[m]$ and the edge set
$\Einf$ given by
\[\Einf(\omega)=\left\{\{i,j\}\mid \sum_{k=1}^{\infty}\left(W_{ij}(k,\omega)
+W_{ji}(k,\omega)\right)=\infty\right\}.\]
We refer to $\Ginf$ as \textit{the infinite flow graph} of $\Wc$.
By the Kolmogorov's 0-1 law, the infinite flow graph of an independent random
chain $\Wc$ is almost surely equal to a deterministic graph.
It has been shown that this determinstic graph
is equal to the infinite flow graph of the expected chain $\{\EXP{W(k)}\}$
(\cite{TouriNedich:Approx}, Theorem~5).

For a matrix $W$, let $W_i$ and $W^j$ denote the $i$th row vector
and the $j$th column vector of $W$, respectively. Also, for a chain $\Wc$, we let
\[W(k:t_0)=W(k)\cdots W(t_0+1)\qquad\hbox{for $k>t_0\geq 0$},\]
with $W(k:k)=I$ for all $k\geq 0$. With these preliminary definitions and notation in place,
we can state the main result of the current study.
 \begin{theorem}\label{thrm:mainresult}
    Let $\Wc$ be an independent random stochastic chain which is balanced and
strongly aperiodic. Then, for any $t_0\geq 0$, the product
$W(k:t_0)=W(k)\cdots W(t_0+1)$ converges to a random stochastic matrix
$W(\infty:t_0)$ almost surely. Furthermore, for all $i,j$ in the same
connected component of the infinite flow graph of $\Wc$, we have
$W_{i}(\infty:t_0)=W_j(\infty:t_0)$ almost surely.
\end{theorem}

To prove Theorem~\ref{thrm:mainresult}, we develop some auxiliary results
in the forthcoming sections, while deferring the proof to the last section.

As an immediate consequence of Theorem~\ref{thrm:mainresult}, it follows that
$W(\infty:t_0)$ has rank at most $\tau$ where $\tau$ is the number of
the connected components of the infinite flow graph $\Ginf$ of $\Wc$.
Thus,
if $\Ginf$ is a connected graph, the limiting random matrix
$W(\infty:t_0)=\lim_{k\to\infty}W(k:t_0)$ is a rank-one
random stochastic matrix almost surely, i.e.
$W(\infty:t_0)=ev^T(t_0)$ almost surely for some stochastic vector $v(t_0)$.
This and Theorem~\ref{thrm:mainresult} imply that: if
an independent random chain $\Wc$ is balanced and strongly aperiodic, then
$\Wc$ is almost surely strongly ergodic (as defined in \cite{SenetaCons})
if and only if the infinite flow graph of $\Wc$ is connected.

\section{Dynamic System Perspective}\label{sec:dynamics}
In order to prove Theorem~\ref{thrm:mainresult}, we establish some intermediate
results, some of which are applicable to a more general category of
random stochastic chains, namely adapted random chains.
For this, let $\Wc$ be a random chain adapted to a filtration $\Fc$. For an
integer $t_0\geq 0$ and a vector $v\in \Rm$,
consider the trivial random vector $x(t_0):\Omega\to\Rm$ defined by
$x(t_0,\omega)=v$ for all $\omega\in\Omega$. Now, recursively define:
  \begin{align}\label{eqn:dynsys}
    x(k+1)=W(k+1)x(k)\qquad\mbox{for all $k\geq t_0$}.
  \end{align}
Note that $x(t_0)$ is measurable with respect to the trivial $\sigma$-algebra
$\{\emptyset,\Omega\}$ and hence, it is measurable with respect to $\F_{t_0}$.
Also, since $\Wc$ is adapted to $\Fc$, it follows that for any $k>t_0$, $x(k)$
is measurable with respect to $\F_{k}$. We refer to $\xc$ as a random dynamics
driven by $\Wc$ started at the initial point $(t_0,v)\in \Rmo$.
We say that a given property holds for any dynamics $\xc$ driven by $\Wc$
if that property holds for any initial point $(t_0,v)\in \Rmo$.

If $\lim_{k\to\infty}W(k:t_0)=W(\infty:t_0)$ exists almost surely,
then the random dynamics $\xc$
converges to $W(\infty:t_0)v$ almost surely  for any initial point $(t_0,v)\in \Rmo$.
Also, note that
for any $i,j\in[m]$, we have
$\lim_{k\to\infty}\|W_i(k,t_0)-W_j(k,t_0)\|=0$ almost surely if and only if
$\lim_{k\to\infty} \left(x_i(k)-x_j(k)\right)=0$ almost surely for any
initial point. When verifying the latter relation, due to the linearity of the
dynamics, it suffices to check that
$\lim_{k\to\infty} \left(x_i(k)-x_j(k)\right)=0$ for all the initial points
of the form $(t_0,e_\ell)$ with $\ell\in [m]$,
where $\{e_1,\ldots,e_m\}$ is the standard basis for $\Rm$.

In order to study the limiting behavior of the products $W(k:t_0)$,
we study the limiting behavior of the dynamics $\xc$ driven by $\Wc$.
This enables us to use the dynamic system's tools
and its stability theory to draw conclusions about the limiting behavior
of the products $W(k:t_0)$.

\subsection{Why the Infinite Flow Graph}
In this section we provide a result showing
the relevance of the infinite flow graph
to the study of the product of stochastic matrices.
Let us consider a deterministic chain $\Ac$ of stochastic matrices and
let us define \textit{mutual ergodicity} and
\textit{an ergodic index} as follows.

 \begin{definition}
For an $m\times m$ chain $\Ac$ of stochastic matrices, we say that
an index $i\in[m]$ is ergodic if $\lim_{k\to\infty}A_i(k:t_0)$
exists for all $t_0\geq 0$.
Also, we say that two disctinct indices $i,j\in[m]$ are mutually ergodic
if $\lim_{k\to\infty}\|A_i(k:t_0)-A_j(k:t_0)\|=0$.
 \end{definition}

From the definition it immediately follows that an index $i\in [m]$ is ergodic
for a chain $\Ac$ if and only if $\lim_{k\to\infty}x_i(k)$ exists for
any dynamics $\xc$ driven by $\Ac$. Similarly, indices $i,j\in[m]$ are
mutually ergodic if and only if $\lim_{k\to\infty}\left(x_i(k)-x_j(k)\right)=0$
for any dynamics driven by $\Ac$.

The following result illustrates the relevance of the infinite flow graph to
the study of the products of stochastic matrices.
\begin{lemma}(\cite{TouriNedich:Approx}, Lemma~2)
Two distinct indices $i,j\in [m]$ are mutually ergodic only if
$i$ and $j$ belong to the same connected component of the infinite flow graph
$\Ginf$ of $\Ac$.
 \end{lemma}

Generally, if $i$ and $j$ are mutually ergodic indices, it is not necessarily
true that they are ergodic indices. As an example,
consider the $4\times 4$ stochastic chain $\Ac$ defined by:
    \begin{align}\nonumber
    A(2k)=
    \left[
      \begin{array}{cccc}
        1&0&0&0\\
        1&0&0&0\\
        1&0&0&0\\
        0&0&0&1
      \end{array}\right],\quad
    A(2k+1)=
    \left[
      \begin{array}{cccc}
        1&0&0&0\\
        0&0&0&1\\
        0&0&0&1\\
        0&0&0&1
      \end{array}\right]\qquad\hbox{for all $k\geq 1$}.
    \end{align}
It can be verified that for any starting time $t_0\geq 0$ and any $k> t_0$,
we have $A(k:t_0)=A(k)$. Thus, it follows that
indices $2$ and $3$ are mutually ergodic, while
$\lim_{k\to\infty}A_{2}(k:t_0)$ and $\lim_{k\to\infty}A_3(k:t_0)$ do not exist.

The following result shows that under special circumstances,
we can assert that some indices are ergodic
if we know that a certain mutual ergodicity pattern exists in a chain.
\begin{lemma}\label{lemma:random:extensionthrm1}
Let $S$ be a connected component of the infinite flow graph $\Ginf$ of a chain
$\Ac$. Suppose that indices $i$ and $j$ are mutually ergodic
for all distinct $i,j\in S$. Then, every index $i\in S$ is ergodic.
     \end{lemma}
     \begin{proof}
Without loss of generality let us assume that $S=\{1,\ldots,i^*\}$ for some
$i^*\in [m]$. Let $\bar S$ be the complement of $S$.
For the given chain $\Ac$ and the connected component $S$,
let the chain $\Bc$ be defined by:
       \begin{align}\nonumber
        B_{ij}(k)=\left\{
        \begin{array}{ll}
        A_{ij}(k)&\mbox{if $i\not=j$ and $i,j\in S$ or $i,j\in \bar{S}$},\\
        0&\mbox{if $i\not=j$ and $i\in S,j\in\bar{S}$ or $i\in\bar{S},j\in{S}$},\\
        A_{ii}(k)+\sum_{\ell\in\bar{S}}A_{i\ell}(k)&\mbox{if $i=j\in S$},\\
        A_{ii}(k)+\sum_{\ell\in{S}}A_{i\ell}(k)&\mbox{if $i=j\in \bar{S}$}.
        \end{array}\right.
        \end{align}
Then, $B(k)$ has the block diagonal structure of the following form
       \[B(k)=\left[\begin{array}{cc}
        B_1(k)&0\\0&B_2(k)
        \end{array}\right]\qquad \mbox{for all $k\geq1$}.\]
By construction the chain $\Bc$ is stochastic.
It can be verified that
$\sum_{k=1}^\infty|A_{ij}(k)-B_{ij}(k)|<\infty$ for all
$i,j\in[m]$. Thus, $\Bc$ is an $\ell_1$-approximation of $\Ac$ as defined
in~\cite{TouriNedich:Approx}. Then, by Lemma~1 in \cite{TouriNedich:Approx}, it
follows that indices $i$ and $j$ are mutually ergodic for the chain
$\Bc$ for all distinct $i,j\in S$. By the block diagonal form of $\Bc$,
it follows that $i$ and $j$ are mutually ergodic
for the $|S|\times |S|$ chain $\{B_1(k)\}$ and all $i,j\in S$.
This, however, implies that the chain $\{B_1(k)\}$ is weakly ergodic
(as defined in \cite{SenetaCons}) and,
as proven in Theorem~1 in \cite{SenetaCons}, this further implies that
$\{B_1(k)\}$ is strongly ergodic, i.e. any index $i\in S$ is ergodic for
$\{B_1(k)\}$.
Again, by the application of Lemma~1 in~\cite{TouriNedich:Approx},
we conclude that any index $i\in S$ is ergodic for $\Ac$.
     \end{proof}


  \subsection{Time-varying Lyapunov Functions}\label{sec:lyapunov}
Here, we show that under general conditions, a rich family of time-varying Lyapunov
functions exists for the dynamics $\xc$ driven by a random chain $\Wc$.

Let us define an absolute probability process for an adapted chain $\Wc$,
which is an extension of the concept of the absolute probability sequence
introduced by A.~Kolmogorov for deterministic chains
in~\cite{KolmogoroffMarkov}.

\begin{definition}\label{def:absoluteprobability}
We say that a random (vector) process $\pc$ is an absolute probability
process for a random chain $\Wc$ adapted to $\Fc$ if
     \begin{enumerate}
       \item the random process $\pc$ is adapted to $\Fc$,
       \item the vector $\pi(k)$ is stochastic almost surely
         for all $k\geq 1$, and
       \item the following relation holds almost surely
  \[\EXP{\pi^T(k+1)W(k+1)\mid \F_{k}}=\pi^T(k)\qquad\mbox{for all $k\geq 0$}.\]
     \end{enumerate}
\end{definition}
When an absolute probability process exists for a chain, we say that
the chain admits an absolute probability process.

For a deterministic chain of stochastic matrices $\Ac$, Kolmogorov showed
in~\cite{KolmogoroffMarkov} that there exists a sequence of stochastic vectors
$\{v(k)\}$ such that $v^T(k+1)A(k+1)=v^T(k)$ for all $k\geq 0$.
Note that, for an independent random chain,
any absolute probability sequence for the expected chain is an
absolute probability process for the random chain.
Thus, the existence of an absolute probability process for an
independent random chain of stochastic matrices follows
immediately from the Kolmogorov's existence result.
As another non-trivial example of random chains that admit an absolute
probability process, one may consider an adapted random chain
$\Wc$ that is doubly stochastic almost surely. In this case,
the static sequence $\{\frac{1}{m}e\}$ is an absolute probability process
for $\Wc$, where $e\in\mathbb{R}^m$ is the vector with all components equal
to 1.

Now, suppose that we have an adapted chain $\Wc$ which admits
an absolute probability sequence $\pc$. Also, let $g:\R\to\R$ be an
arbitrary convex function. Let us define
the function $V_{g,\pi}:\re^m\times\mathbb{Z}^+\to\mathbb{R}$, as follows:
\begin{align}\label{eqn:Vgn}
     V_{g,\pi}(x,k)=\sum_{i=1}^{m}\pi_i(k)g(x_i)-g(\pi^T(k)x)
  \qquad \mbox{for all $x\in\mathbb{R}^m$ and all $k\geq 0$}.
   \end{align}
   From the definition of an absolute probability process, it follows that
$V_{g,\pi}(x(k),k)$ is measurable with respect to $\F_k$ for any
dynamics $\xc$ driven by a chain $\Wc$ that is adapted to $\Fc$.
Also, since $\pi(k)$ is almost surely stochastic vector and
$g$ is a convex function, it follows that for any $x\in \Rm$,
we have $V_{g,\pi}(x,k)\geq 0$ almost surely for all $k\geq 0$.


Next, we show that $V_{g,\pi}$ is a time-varying Lyapunov function for
the dynamics~\eqref{eqn:dynsys} for any convex function $g$. In particular,
we prove that
$\{V_{g,\pi}(x(k),k)\}$ is a super-martingale sequence irrespective of
the initial point for the dynamics $\xc$.
 \begin{theorem}\label{thrm:compairsonfunctions}
 Let $\Wc$ be an adapted chain that admits an absolute probability process
$\pc$. Then, for the dynamics~\eqref{eqn:dynsys} started at
any initial point $(t_0,v)\in\Rmo$, we have
 \[\EXP{V_{g,\pi}(x(k+1),k+1)\mid \F_k}\leq V_{g,\pi}(x(k),k)
  \qquad\mbox{for all $k\ge t_0$}.\]
 \end{theorem}
 \begin{proof}
    By the definition of $V_{g,\pi}$ in \eqref{eqn:Vgn},
we have almost surely
    \begin{align}\label{eqn:infflowstable:samplepathdecrease}
    V_{g,\pi}(x(k+1),k+1)&
=\sum_{i=1}^m\pi_i(k+1)g(x_i(k+1))-g(\pi^T(k+1)x(k+1))\cr
    &=\sum_{i=1}^m\pi_i(k+1)g([W(k+1)x(k)]_i)-g(\pi^T(k+1)x(k+1))\cr
    &\leq \sum_{i=1}^m\pi_i(k+1)
\sum_{j=1}^mW_{ij}(k+1)g(x_j(k))-g(\pi^T(k+1)x(k+1)),
    \end{align}
where in the second equality we use $[\cdot]_i$ to denote the $i$th component
of a vector, while the inequality is obtained
by using the convexity of $g(\cdot)$ and
the fact that matrix $W(k)$ is stochastic almost surely.
Since $\pc$ is an absolute probability process for $\Wc$,
it follows that $\EXP{\pi^T(k+1)W(k+1)\mid \F_k}=\pi^T(k)$.
Also, since $x(k)$ is measurable with respect to $\F_k$, by taking
the conditional expectation with respect to $\F_k$ on both sides of
Eq.~\eqref{eqn:infflowstable:samplepathdecrease}, we obtain almost surely
    \begin{align*}
    \EXP{V_{g,\pi}(x(k+1),k+1)\mid \F_k}
 &\leq \sum_{j=1}^m\pi_j(k)g(x_j(k))-\EXP{g(\pi^T(k+1)x(k+1))\mid \F_k}\cr
    & \leq \sum_{j=1}^m\pi_j(k)g(x_j(k))-g(\EXP{\pi^T(k+1)x(k+1)\mid \F_k}),
    \end{align*}
    where the last inequality follows by the convexity of $g$ and
Jensen's inequality. The result follows by using $x(k+1)=W(k+1)x(k)$
and the definition of absolute probability process.
 \end{proof}

Theorem~\ref{thrm:compairsonfunctions} shows that
the dynamics~\eqref{eqn:dynsys} admits infinitely many time-varying Lyapunov functions,
provided that $\Wc$ admits an absolute probability process.

Since $V_{g,\pi}(x(k),k)\geq 0$ almost surely for all $k\geq 0$, it follows
that $\{V_{g,\pi}(x(k),k)\}$ is a bounded super-martingale.
Hence, it is convergent almost surely irrespective of the initial point of
the dynamics $\{x(k)\}$ and the choice of the convex function $g$.
  \begin{corollary}
    Let $\Wc$ be an adapted chain that admits an absolute probability process
$\pc$. Then, for any dynamics $\xc$ driven by $\Wc$ and for any convex
function $g:\mathbb{R}\to\mathbb{R}$,
the limit $\lim_{k\to\infty}V_{g,\pi}(x(k),k)$ exists almost surely.
  \end{corollary}

  \subsection{Time-varying Quadratic Lyapunov Function}
  In the sequel, we focus on the particular choice of
function $g(s)=s^2$ in relation \eqref{eqn:Vgn}. For convenience, we let
  \[V_{\pi}(x,k)=\sum_{i=1}^m\pi_i(k)(x_i-\pi^T(k)x)^2
  =\sum_{i=1}^m\pi_i(k)x^2_i-(\pi^T(k)x)^2.\]
For this function,
we can provide a lower bound for the decrease of the conditional expectations
$\EXP{V_g(x(k+1),k+1)\mid \F_k}$, which is exact under certain conditions.

  \begin{theorem}\label{thrm:quaddecrease}
    Let $\Wc$ be an adapted random chain with an absolute probability process
$\pc$. Then, for any dynamics $\xc$ driven by $\Wc$, we have almost surely
    \begin{align*}
      \EXP{V_\pi(x(k+1),k+1)\mid \F_k}\leq V_{\pi}(x(k),k)
     -\sum_{i<j}H_{ij}(k)(x_i(k)-x_j(k))^2\qquad\hbox{for all $k\ge t_0$,}
    \end{align*}
where
$H(k)=\EXP{W^T(k+1)\diag(\pi(k+1))W(k+1)\mid \F_k}$ with $\diag(v)$
denoting the diagonal matrix induced by a vector $v$
(i.e., with components $v_i$ on the main diagonal), and
$\sum_{i<j}=\sum_{i=1}^m\sum_{j=i+1}^m$.
Furthermore, if $\pi^T(k+1)W(k+1)=\pi^T(k)$ almost surely, then the
inequality holds as an equality.
  \end{theorem}
  \begin{proof}
    We have for all $k\ge t_0$,
    \begin{align}\label{eqn:infflowstable:differentform}
    V_{\pi}(x(k),k)
    =\sum_{i=1}^{m}\pi_i(k)x_i^2(k)-(\pi^T(k)x(k))^2
   =x^T(k)\diag(\pi(k))x(k)-(\pi^T(k)x(k))^2.
    \end{align}
 Thus, by letting $\Delta(x(k),k)=V_{\pi}(x(k),k)-V_{\pi}(x(k+1),k+1)$
and using $x(k+1)=W(k+1)x(k)$, we obtain for all $k\ge t_0$,
\begin{align}\nonumber
\Delta(x(k),k)&=x^T(k)\diag(\pi(k))x(k)-(\pi^T(k)x(k))^2\cr
&\qquad-\left\{x^T(k+1)\diag(\pi(k+1))x(k+1)-(\pi^T(k+1)x(k+1))^2\right\}\cr
&=x^T(k)\left[\diag(\pi(k))-W^T(k+1)\diag(\pi(k+1))W(k+1)\right]x(k)\cr
&\qquad+\left\{
(\pi^T(k+1)x(k+1))^2-(\pi^T(k)x(k))^2\right\}\cr
&=x^T(k)L(k)x(k)+\left\{
(\pi^T(k+1)x(k+1))^2-(\pi^T(k)x(k))^2\right\},
\end{align}
where $L(k)=\diag(\pi(k))-W^T(k+1)\diag(\pi(k+1))W(k+1)$.

Note that the sequence $\{\pi^T(k)x(k)\}$ is a martingale,
implying that $\{-(\pi^T(k)x(k))^2\}$ is a super-martingale.
Thus, by taking the conditional expectation on both sides of the
preceding equality and noticing that $x(k)$ is measurable with respect to
$\F_k$, we have almost surely
\begin{align}\label{eqn:infflowstable:Delta}
\EXP{\Delta(x(k),k)\mid\F_k}&\geq \EXP{x^T(k)L(k)x(k)\mid\F_k}
=x^T(k)\EXP{L(k)\mid\F_k}x(k)\quad\mbox{for all $k\ge t_0$.}
\end{align}

Further, letting $e\in\mathbb{R}^m$ be the vector with all
components equal to 1, from the definition of $L(k)$ we almost surely have
for all $k\ge t_0$:
\begin{align}\nonumber
\EXP{L(k)\mid\F_k}e
&=\EXP{\diag(\pi(k))e-W^T(k+1)\diag(\pi(k+1))W(k+1)e\mid\F_k}\cr
&=\pi(k)-\EXP{W^T(k+1)\pi(k+1)\mid \F_k}=0,
\end{align}
which holds since $W(k)$ is stochastic almost surely and
$\pc$ is an absolute probability process for $\Wc$.
Thus, the random matrix $\EXP{L(k)\mid\F_k}$ is symmetric and
$\EXP{L(k)\mid\F_k}e=0$ almost surely.
It can be shown that for a symmetric matrix $A$ with $Ae=0$,
we have $x^TAx=-\sum_{i<j}A_{ij}(x_i-x_j)^2$. Then, it follows that
almost surely
\[x^T(k)\EXP{L(k)\mid\F_k}x(k)=-\sum_{i<j}H_{ij}(k)(x_i(k)-x_j(k))^2,\]
where $H(k)=\EXP{W^T(k+1)\diag(\pi(k+1))W(k+1)\mid \F_k}$.
Using this relation in inequality~\eqref{eqn:infflowstable:Delta}, we conclude
that almost surely
\begin{align}\label{eqn:jed}
\EXP{V_{\pi}(x(k+1),k+1)\mid \F_k}
\leq V_{\pi}(x(k),k)-\sum_{i<j}H_{ij}(k)(x_i(k)-x_j(k))^2
\quad\mbox{for all $k\ge t_0$}.\end{align}

In the proof of inequality~\eqref{eqn:jed}, the inequality sign appears due to
relation~\eqref{eqn:infflowstable:Delta} only.
If $\pi^T(k+1)W(k+1)=\pi^T(k)W(k)$ almost surely, then we have $\pi^T(k+1)x(k+1)=\pi^T(k)x(k)$ almost surely. Thus,
relation~\eqref{eqn:infflowstable:Delta} holds as an equality and, consequently,
so does relation~\eqref{eqn:jed}.
\end{proof}

One of the important implications of
Theorem~\ref{thrm:quaddecrease} is the following result.
\begin{corollary}\label{cor:infflowstable:sumfinite}
Let $\Wc$ be an adapted random chain that admits
an absolute probability process $\pc$.
Then, for any random dynamics $\xc$ driven by $\Wc$,
we have for all $t_0\geq 0$,
\[\EXP{\sum_{k=t_0}^{\infty}\sum_{i<j}L_{ij}(k)\left(x_i(k)-x_j(k)\right)^2}
\leq \EXP{V_{\pi}(x(t_0),t_0)}<\infty,\]
where $L(k)=W^T(k+1)\diag(\pi(k+1))W(k+1)$.
\end{corollary}
\begin{proof}
By taking expectation on both sides of the relation in
Theorem~\ref{thrm:quaddecrease}, 
we obtain for all $k\ge t_0$:
\begin{align}\label{eqn:infflowstbale:appexpect}
\EXP{V_{\pi}(x(k+1),k+1)}\leq \EXP{V_{\pi}(x(k),k)}-\EXP{\sum_{i<j}
\EXP{L_{ij}(k)\mid \F_k}(x_i(k)-x_j(k))^2}.
\end{align}
Since $x(k)$ is measurable with respect to $\F_k$, it follows that
\[\EXP{L_{ij}(k)\mid \F_k}(x_i(k)-x_j(k))^2
=\EXP{L_{ij}(k)(x_i(k)-x_j(k))^2\mid \F_k}=\EXP{L_{ij}(k)(x_i(k)-x_j(k))^2}.\]
Using this relation in Eq.~\eqref{eqn:infflowstbale:appexpect}, we see that
for all $k\ge t_0$:
\begin{align}\nonumber
\EXP{V_{\pi}(x(k+1),k+1)}\leq \EXP{V_{\pi}(x(k),k)}
-\EXP{\sum_{i<j}L_{ij}(k)(x_i(k)-x_j(k))^2}.
\end{align}
Hence, $\sum_{k=t_0}^\infty\EXP{\sum_{i<j}L_{ij}(k)(x_i(k)-x_j(k))^2}
\leq \EXP{V_\pi(x(t_0),t_0)}$ for any $t_0\geq 0$.
\end{proof}

\section{Class $\Pstar$}\label{sec:infflowstable:Pstar}
In this section, we introduce a class of random chains, which we refer to as
the \textit{class $\Pstar$}, and we prove one of the central results of this work.
In particular, we show that the claim of Theorem~\ref{thrm:mainresult} holds
for any chain that is in the class $\Pstar$ and
satisfies some form of aperiodicity.

\begin{definition}\label{def:infflowstable:Pstar}
The class $\Pstar$ is the class of random adapted chains that admit
an absolute probability process $\pc$ which is uniformly bounded away
from zero almost surely, i.e., $\pi_i(k)\geq p^*$ almost surely
for some scalar $p^*>0$, and for all $k\geq 0$ and all $i\in[m]$.
We write this concisely as $\pc\geq p^*>0$.
\end{definition}
It may appear that the definition of the class $\Pstar$ is a rather
restrictive. Later on, we show that in fact the class $\Pstar$
contains a broad family of deterministic and random chains.

To establish the main result of this section, we make use of the following
intermediate result.

\begin{lemma}\label{lemma:infflowstable:suminfinity}
Let $\Ac$ be a deterministic chain with the infinite flow graph
$\Ginf=([m],\Einf)$. Let $(t_0,v)\in\Rmo$ be an initial point for the dynamics
driven by $\Ac$. If
\[\lim_{k\to\infty}\left(x_{i_0}(k)-x_{j_0}(k)\right)\not=0,\]
 for some $i_0,j_0$ belonging to the same connected component of $\Ginf$,
then we have
\[\sum_{k=t_0}^{\infty}\sum_{i<j}[(A_{ij}(k+1)+A_{ji}(k+1))(x_i(k)-x_j(k))^2]
=\infty.\]
\end{lemma}
\begin{proof}
Let $i_0$ and $j_0$ be in the same connected component of $\Ginf$ and such that
\[\limsup_{k\to\infty}\left(x_{i_0}(k)-x_{j_0}(k)\right)=\al>0.\]
Without loss of generality we may assume that $x(t_0)\in[-1,1]^m$, for
otherwise we can consider the dynamics started at
$y(t_0)=\frac{1}{\|x(t_0)\|_{\infty}}x(t_0)$. Let $S$ be the vertex set
of the connected component in $\Ginf$ containing
$i_0,j_0$, and without loss of generality assume that $S=\{1,2,\ldots,q\}$
for some $q\in[m]$, $q\ge 2$.
Then, by the definition of the infinite flow graph, there exists a large enough
$K\geq t_0$ such that
\[\sum_{k=K}^{\infty}A_{S}(k+1)\leq \frac{\al}{32q},\]
where $A_{S}(k+1)=A_{S\bar{S}}(k+1)+A_{\bar{S}S}(k+1)$.
Furthermore, since $\limsup_{k\to\infty}\left(x_{i_0}(k)-x_{j_0}(k)\right)
=\al>0$, there exists a time instance $t_1\geq K$ such that
$x_{i_0}(t_1)-x_{j_0}(t_1)\geq\frac{\al}{2}$.

Let $\sigma:[q]\to[q]$ be a permutation such that
$x_{\sigma(1)}(t_1)\geq x_{\sigma(2)}(t_1)\geq \cdots
\geq x_{\sigma(q)}(t_1)$, i.e.\ $\sigma$ is an ordering of
$\{x_i(t_1)\mid i\in[q]\}$. Since $x_{i_0}(t_1)-x_{j_0}(t_1)
\geq \frac{\al}{2}$, it follows that $x_{\sigma(1)}(t_1)-x_{\sigma(q)}(t_1)
\geq \frac{\al}{2}$ and, therefore, there exists $\ell\in[q]$ such that
$x_{\sigma(\ell)}(t_1)-x_{\sigma(\ell+1)}(t_1)\geq \frac{\al}{2q}$. Let
\[T_1=\arg\min_{t>t_1}\sum_{k=t_1}^{t}\sum_{\substack{i,j\in [q]\\i
\leq \ell,\ell+1\leq j}}(A_{\sigma(i)\sigma(j)}(k+1)
+A_{\sigma(j)\sigma(i)}(k+1))\geq \frac{\al}{32q}.\]
Since $S$ is a connected component of the infinite flow graph $\Ginf$,
we must have $T_1<\infty$; otherwise, $S$ could be decomposed into
two disconnected components $\{\sigma(1),\ldots,\sigma(l)\}$ and
$\{\sigma(l+1),\ldots,\sigma(q)\}$.

Now, let $R=\{\sigma(1),\ldots,\sigma(l)\}$.
We have for any $k\in [t_1,T_1]$:
\begin{align}\nonumber
\sum_{k=t_1}^{T_1-1}A_{R}(k+1)
&=\sum_{k=t_1}^{T_1-1}\left(\sum_{\substack{i,j\in [q]\\i\
leq \ell,\ell+1\leq j}}(A_{\sigma(i)\sigma(j)}(k+1)
+A_{\sigma(j)\sigma(i)}(k+1))\right.\cr
&\qquad\left.+\sum_{i\leq \ell,j\in \bar{S}}A_{\sigma(i)j}(k+1)
+\sum_{i\in \bar{S},j\leq l}A_{i\sigma(j)}(k+1)\right)\cr
&\leq
\sum_{k=t_1}^{T_1-1}\sum_{\substack{i,j\in [q]\\i
\leq \ell,\ell+1\leq j}}(A_{\sigma(i)\sigma(j)}(k+1)
+A_{\sigma(j)\sigma(i)}(k+1))+\sum_{k=K}^{\infty}A_{S}(k+1)
\leq \frac{\al}{16q},
\end{align}
which follows by the definition of $T_1$ and the choice of $t_1\geq K$.
By Lemma~1 in~\cite{ErgodicityPaper}, it follows that
for $k\in [t_1,T_1]$,
\[\max_{i\in R}x_i(k)\leq \max_{i\in R}x_i(t_1)+2\frac{\al}{16q},
\quad\quad
\min_{i\in S\setminus R}x_i(k)
\geq \min_{i\in S\setminus R}x_i(t_1)-2\frac{\al}{16q}.\]
Thus, for any $i,j\in[q]$ with $i\leq l$ and $j\geq l+1$, and for any
$k\in [t_1,T_1]$, we have
 \[x_{\sigma(i)}(k)-x_{\sigma(j)}(k)
\geq 2 \left(2\frac{\al}{16q}\right)=\frac{\al}{4q}.\]
 Therefore,
\begin{align}\nonumber
&\sum_{k=t_1}^{T_1}\sum_{\substack{i,j\in[q]\\ i
\leq \ell,\ell+1\leq j}}(A_{\sigma(i)\sigma(j)}(k+1)
+A_{\sigma(j)\sigma(i)}(k+1))(x_{\sigma(i)}(k)-x_{\sigma(j)}(k))^2\cr
&\qquad\geq
(\frac{\al}{4q})^2\sum_{k=t_0}^{T_1}\sum_{\substack{i,j\in[q]\\ i\leq l,j
\geq l+1}}(A_{\sigma(i)\sigma(j)}(k+1)+A_{\sigma(j)\sigma(i)}(k+1))
\geq \left(\frac{\al}{4q}\right)^2\frac{\al}{32q}=\beta>0.
\end{align}
Further, it follows that:
\begin{align}\nonumber
&\sum_{k=t_1}^{T_1}\sum_{i<j}(A_{ij}(k+1)+A_{ji}(k+1))(x_i(k)-x_j(k))^2\cr
&\qquad \geq
\sum_{k=t_1}^{T_1}\sum_{\substack{i,j\in[q]\\ i
\leq \ell,\ell+1\leq j}}(A_{\sigma(i)\sigma(j)}(k+1)
+A_{\sigma(j)\sigma(i)}(k+1))(x_{\sigma(i)}(k)-x_{\sigma(j)}(k))^2\geq\beta.
\end{align}

Since $\limsup_{k\to\infty}\left(x_{i_0}(k)-x_{j_0}(k)\right)=\al>0$, there
exists a time $t_2>T_1$ such that $x_{i_0}(t_2)-x_{j_0}(t_2)
\geq \frac{\al}{2}$. Then, using the above argument, there exists $T_2>t_2$
such that $\sum_{k=t_2}^{T_2}\sum_{i<j}(A_{ij}(k+1)
+A_{ji}(k+1))(x_i(k)-x_j(k))^2\geq\beta$. Hence, using the induction,
we can find time instances
\[\cdots>T_{\xi+1}>t_{\xi+1}>T_\xi>t_\xi>T_{\xi-1}>t_{\xi-1}>\cdots>T_1>t_1
\ge t_0,\]
such that $\sum_{k=t_\xi}^{T_\xi}\sum_{i<j}(A_{ij}(k+1)
+A_{ji}(k+1))(x_i(k)-x_j(k))^2\geq \beta$ for any $\xi\geq 1$.
The intervals $[t_{\xi},T_\xi]$ are non-overlapping subintervals of
$[t_0,\infty)$, implying that
\[\sum_{k=t_0}^{\infty}\sum_{i<j}(A_{ij}(k+1)+A_{ji}(k+1))(x_i(k)-x_j(k))^2
=\infty.\]
\end{proof}

For our main result,
let us define the weak aperiodicity for an adapted random chain.
 \begin{definition}
   We say that an adapted random chain $\Wc$ is weakly aperiodic if
for some $\gamma>0$, and for all distinct $i,j\in[m]$ and all $k\ge0$,
   \[\EXP{W^{iT}(k+1)W^j(k+1)\mid \F_k}\geq \gamma
\EXP{W_{ij}(k+1)+W_{ji}(k+1)\mid \F_k}.\]
 \end{definition}
 Now, we establish the main result of this section.
\begin{theorem}\label{thrm:infflowstability:infiniteflowstability}
Let $\Wc\in\Pstar$ be an adapted chain that is weakly aperiodic.
Then, $\lim_{k\to\infty}W(k:t_0)=W(\infty:t_0)$ exists almost surely
for any $t_0\geq 0$. Moreover, the event under which
$W_i(\infty:t_0)=W_j(\infty:t_0)$ for all $t_0\geq 0$ is almost surely equal
to the event that $i,j$ are belonging to the same connected component of
the infinite flow graph of $\Wc$.
\end{theorem}
\begin{proof}
Since $\Wc$ is in $\Pstar$, $\Wc$ admits an absolute probability process $\pc$
such that $\pc\geq p^*>0$ almost surely. Thus, it follows that
\[p^{*}\EXP{W^T(k+1)W(k+1)\mid \F_k}\leq
\EXP{W^T(k+1)\diag(\pi(k+1))W(k+1)\mid \F_k}=H(k+1).\]
On the other hand, by the weak aperiodicity, we have
\[\gamma \EXP{W_{ij}(k+1)+W_{ji}(k+1)\mid \F_k}
\leq \EXP{W^{iT}(k+1)W^j(k+1)\mid \F_k},\]
for some $\gamma\in (0,1]$ and for all distinct $i,j\in [m]$. Thus, we have
$p^*\gamma\EXP{W_{ij}(k+1)+W_{ji}(k+1)\mid \F_k}\leq H_{ij}(k+1)$. By
Corollary~\ref{cor:infflowstable:sumfinite},
for the random dynamics $\xc$ driven by $\Wc$ and started at arbitrary
$(t_0,v)\in\Rmo$, it follows that
\[p^*\gamma\sum_{k=t_0}^\infty\EXP{\sum_{i<j}(W_{ij}(k)
+W_{ji}(k))(x_i(k)-x_j(k))^2}\leq \EXP{V_\pi(x(t_0),t_0)}.\]
As a consequence,
\[\sum_{k=t_0}^\infty\sum_{i<j}(W_{ij}(k)+W_{ji}(k))(x_i(k)-x_j(k))^2
<\infty\qquad\mbox{almost surely}.\]
Therefore, by Lemma~\ref{lemma:infflowstable:suminfinity}, we conclude that
$\lim_{k\to\infty}\left(x_i(k,\omega)-x_j(k,\omega)\right)=0$ for any $i,j$
belonging to the same connected component of $\Ginf(\omega)$, for almost all
$\omega\in \Omega$. By Lemma~\ref{lemma:random:extensionthrm1} it follows
that every index $i\in[m]$ is ergodic for almost all $\omega \in \Omega$.
By considering the initial conditions $(t_0,e_\ell)\in\Rmo$ for all
$\ell\in[m]$, the assertion follows.
\end{proof}

Theorem~\ref{thrm:infflowstability:infiniteflowstability}
shows that the dynamics in~\eqref{eqn:dynsys} is convergent
almost surely for aperiodic chains $\Wc\in \Pstar$. Moreover, the theorem
also characterizes the limiting points of such a dynamics as well
as the limit matrices of the products $W(k:t_0)$ as $k\to\infty$.

\section{Balanced Chains}\label{sec:infflowstable:balanced}
In this section, we characterize a subclass of $\Pstar$ chains,
namely the class of {\it strongly aperiodic balanced chains}.
We first show that this class includes many of the chains that have been
studied in the existing literature. Then, we prove that any aperiodic
balanced chain belongs to the class $\Pstar$. We also show
that a balanced independent random chain is strongly aperiodic, thus
concluding Theorem~\ref{thrm:mainresult}.

Before continuing our analysis on balanced chains,
let us discuss some of the well-known subclasses of such chains:
 \begin{enumerate}
   \item \textbf{Balanced Bidirectional Chains}:
We say that an independent chain $\Wc$ is a balanced bidirectional chain
if there exists some $\al>0$ such that
$\EXP{W_{ij}(k)}\geq \al \EXP{W_{ji}(k)}$ for all $k\geq 1$ and $i,j\in [m]$.
These chains are in fact balanced, since for any $S\subset[m]$ we have:
\[\EXP{W_{S\bar{S}}(k)}=\EXP{\sum_{i\in S,j\in\bar{S}}W_{ij}(k)}
\geq \EXP{\sum_{i\in S,j\in\bar{S}}\al W_{ji}(k)}=\al \EXP{W_{\bar{S}S}(k)}.\]

Examples of such chains are bounded bidirectional deterministic chains,
which are the chains such that $A_{ij}(k)>0$ implies
$A_{ji}(k)>0$ for all $i.j\in[m]$ and all $k\ge1$,
and the positive entries are uniformly bounded from below
by some $\gamma>0$ (i.e., $A_{ij}(k)>0$ implies
$A_{ij}(k)\geq \gamma$ for all $i,j\in[m]$ and all $k\ge1$).
In this case, for $A_{ij}(k)>0$, we have $A_{ij}(k)\geq \gamma
\geq \gamma A_{ji}(k)$ and for $A_{ij}(k)=0$, we have $A_{ji}(k)=0$ and,
hence, in either of the cases $A_{ij}(k)\geq \gamma A_{ji}(k)$.
Therefore, bounded bidirectional chains are examples of balanced
bidirectional chains. Such chains have been considered
in~\cite{Lorenz,multiagent,CaoMora} and, among others, include the Hegselman-Krause model for opinion dynamics~\cite{Krause:97,Krause:02}.
    \item \textbf{Chains with Common Steady State $\pi>0$}:
This ensemble consists of independent random chains $\Wc$
such that $\EXP{\pi^TW(k)}=\EXP{\pi^T(k)}$ for some stochastic vector $\pi>0$
and all $k\geq 1$, which are generalizations of doubly stochastic chains, where
we have $\pi=\frac{1}{m}e$ ($e$ is a vector of ones).
Doubly stochastic chains and the chains with a common steady state
$\pi>0$ have been studied
in~\cite{Olshevsky09a,ErgodicityPaper,TouriNedich:Approx}.
 \end{enumerate}
     To show that a chain with a common steady state $\pi>0$ is a
balanced chain, let us prove the following lemma.
    \begin{lemma}\label{lemma:commonsteadybalance}
    Let $A$ be a stochastic matrix and $\pi>0$ be a stochastic left-eigenvector
of $A$ corresponding to the unit eigenvalue, i.e., $\pi^TA=\pi^T$. Then,
$A_{S\bar{S}}\geq \frac{\pi_{\min}}{\pi_{\max}}A_{\bar{S}S}$ for any
non-trivial $S\subset[m]$, where $\pi_{\max}=\max_{i\in[m]}\pi_i$ and
$\pi_{\min}=\min_{i\in [m]}\pi_i$.
    \end{lemma}
    \begin{proof}
    Let $S\subset[m]$. Since $\pi^TA=\pi^T$, we have
\begin{align}\label{eqn:commonsteady1}
    \sum_{j\in S}\pi_j=\sum_{i\in [m],j\in S}\pi_{i}A_{ij}
=\sum_{i\in S,j\in S}\pi_{i}A_{ij}+\sum_{i\in \bs,j\in S}\pi_{i}A_{ij}.
    \end{align}
    On the other hand, since $A$ is a stochastic matrix,
we have $\pi_i\sum_{j\in[m]}A_{ij}=\pi_i$. Therefore,
    \begin{align}\label{eqn:commonsteady2}
    \sum_{i\in S}\pi_i=\sum_{i\in S}\pi_i\sum_{j\in[m]}A_{ij}
=\sum_{i\in S,j\in S}\pi_iA_{ij}+\sum_{i\in S,j\in \bs}\pi_iA_{ij}.
    \end{align}
    Comparing Eq.~\eqref{eqn:commonsteady1} and Eq.~\eqref{eqn:commonsteady2},
we see that
$\sum_{i\in \bs,j\in S}\pi_{i}A_{ij}=\sum_{i\in S,j\in \bs}\pi_iA_{ij}$.
Therefore,
    \begin{align}\nonumber
    \pi_{\min}A_{\bs S}\leq \sum_{i\in \bs,j\in S}\pi_{i}A_{ij}
=\sum_{i\in S,j\in \bs}\pi_iA_{ij}\leq \pi_{\max}A_{S\bs}.
    \end{align}
    Hence, we have $A_{S\bs}\geq \frac{\pi_{\min}}{\pi_{\max}}A_{\bs S}$
for any non-trivial  $S\subset[m]$.
    \end{proof}

The above lemma  shows that a chain with a common steady state $\pi>0$
is balanced with balancedness coefficient $\al=\frac{\pi_{\min}}{\pi_{\max}}$.
In fact, the lemma yields a much more general result, as provided below.
    \begin{theorem}\label{thrm:infflowstable:generalsteadystate}
    Let $\Wc$ be an independent random chain with a sequence $\pc$ of
stochastic left-eigenvectors for the expected chain corresponding to the unit
eigenvalue, i.e., $\pi^T(k)\EXP{W(k)}=\pi^T(k)$ for all $k\ge 1$.
If $\pc\geq p^*$ for some scalar $p^*>0$,
then $\Wc$ is a balanced chain with a balancedness coefficient
$\al=\frac{p^*}{1-(m-1)p^*}$.
    \end{theorem}
    \begin{proof}
Since $\pi^T(k)\EXP{W(k)}=\pi^T(k)$ for all $k\ge1$, by
Lemma~\ref{lemma:commonsteadybalance} we have
\[\EXP{W_{S\bar{S}}(k)}\geq
\frac{\pi_{\min}(k)}{\pi_{\max}(k)}\EXP{W_{\bar{S}S}(k)}
\qquad\hbox{for any non-trivial $S\subset[m]$ and all $k\geq 1$},\]
By $\pc\geq p^*>0$, it follows that$\pi_{\min}(k)\geq p^*$ for all $k\geq 1$.
Since $\pi(k)$ is a stochastic vector, it further follows
$\pi_{\max}(k)\leq 1-(m-1)\pi_{\min}(k)\leq 1-(m-1)p^*$. Therefore,
for all $k\ge1$,
    \[\EXP{W_{S\bar{S}}(k)}\geq \frac{\pi_{\min}(k)}{\pi_{\max}(k)}
\EXP{W_{\bar{S}S}(k)}\geq \frac{p^*}{1-(m-1)p^*}\EXP{W_{\bar{S}S}(k)},\]
    for any non-trivial $S\subset[m]$. Thus, $\Wc$ is balanced with
a balancedness coefficient $\al=\frac{p^*}{1-(m-1)p^*}$.
    \end{proof}


Theorem~\ref{thrm:infflowstable:generalsteadystate} not only characterizes a
class of balanced chains, but it also provides an alternative characterization
of the balancedness for these chains. Thus, instead of verifying
Definition~\ref{eqn:balanceddef} for every nontrivial subset $S\subset[m]$,
for balancedness of independent random chains,
it suffices to find a sequence $\pc$ of stochastic (unit) left-eigenvectors
of the expected chain $\{\EXP{W(k)}\}$ such that the entries of the sequence
do not vanish as time goes to infinity.

\subsection{Absolute Probability Sequence for Balanced Chains}
In this section, we show that any independent random chain that is strongly
aperiodic and balanced must be in the class $\Pstar$.
The road map to prove this result is as follows: we first show that this result
holds for deterministic chains with uniformly bounded positive entries.
Then, using this result and geometric properties of the set of strongly
aperiodic balanced chains, we prove the statement for deterministic chains,
which immediately implies the result for independent random chains.
To show the result for deterministic chains with uniformly bounded positive entries, we employ the
technique that is used to prove Proposition~4 in
\cite{Lorenz}. However, the argument given in \cite{Lorenz} needs some
extensions to fit in our more general assumption of balanced-ness.

Let $\Ac$ be a deterministic chain
of stochastic matrices. Let $S_j(k)$ be the set of indices  corresponding to
the positive entries in the $j$th column of $A(k:0)$, i.e.,
\[S_j(k)=\{\ell\in [m]\mid A_{\ell j}(k:0)>0\}\qquad
\hbox{for all $j\in [m]$ and all $k\geq 0$}.\]
Also, let $\mu_j(k)$ be the minimum value of these positive entries, i.e.,
\[\mu_j(k)=\min_{\ell\in S_{j}(k)}A_{\ell j}(k:0)>0.\]
\begin{lemma}\label{lemma:infflowstable:boundedcolumns}
Let $\Ac$ be a strongly aperiodic balanced chain such
that the positive entries in each $A(k)$ are  uniformly bounded from below
by a scalar $\gamma>0$. Then, $S_j(k)\subseteq S_j(k+1)$ and
$\mu_j(k)\geq \gamma^{|S_j(k)|-1}$ for all $j\in [m]$ and $k\geq 0$.
\end{lemma}
\begin{proof}
Let $j\in[m]$ be arbitrary but fixed. By induction on $k$, we prove that
$S_j(k)\subseteq S_j(k+1)$ for all $k\ge0$ as well as the desired relation for
$\mu_j(k)$. For $k=0$, we have $A(0:0)=I$ by the definition, so $S_j(0)=\{j\}$.
Then, $A(1:0)=A(1)$ and by the strongly aperiodic assumption on the chain
$\Ac$ we have $A_{jj}(1)\ge\gamma$, implying $\{j\}=S_j(0)\subseteq S_j(1)$.
Furthermore, we have $|S_j(0)|-1=0$ and $\mu_j(0)=1=\gamma^0$.
Hence, the claim is true for $k=0$.

Now suppose that the claim is true for some $k\geq 0$, and consider $k+1$.
Then, for any $i\in S_j(k)$, we have:
\[A_{ij}(k+1:0)=\sum_{\ell=1}^{m}A_{i\ell}(k+1)A_{\ell j}(k:0)
\ge A_{ii}(k+1)A_{ij}(k:0)\ge\gamma\mu_j(k)>0.\]
Thus, $i\in S_j(k+1)$, implying $S_{j}(k)\subseteq S_j(k+1)$.

To show the relation for $\mu_j(k+1)$,
we consider two cases:\\
\textit{Case $A_{S_j(k)\bar{S}_j(k)}(k+1)=0$:}
In this case for any $i\in S_j(k)$, we have:
\begin{align}\label{eqn:one}
A_{ij}(k+1:0) 
&=\sum_{\ell\in S_j(k)} A_{i\ell}(k)A_{\ell j}(k:0)\geq
\mu_j(k)\sum_{\ell\in S_j(k)}A_{i\ell}(k+1)=\mu_j(k),
\end{align}
where the inequality follows
from $i\in S_j(k)$ and $A_{S_j(k)\bar{S}_j(k)}(k+1)=0$,
and the definition of $\mu_j(k)$.
Furthermore, by the balancedness of $A(k)$ and $A_{S_j(k)\bar{S}_j(k)}(k+1)=0$,
it follows that
$0=A_{S_j(k)\bar{S}_j(k)}(k+1)\geq \al A_{\bar{S}_j(k)S_j(k)}(k+1)\geq 0$.
Hence,
$A_{\bar{S}_j(k)S_j(k)}(k+1)=0$. Thus, for any $i\in \bar{S}_j(k)$, we have
\[A_{ij}(k+1:0)=\sum_{\ell=1}^{m}A_{i\ell}(k+1)A_{\ell j}(k:0)
=\sum_{\ell\in \bar{S}_j(k)} A_{i\ell}(k+1)A_{\ell j}(k:0)=0,\]
where the second equality follows from $A_{\ell j}(k:0)=0$ for all
$\ell\in \bar{S}_j(k)$. Therefore, in this case we have $S_j(k+1)=S_j(k)$,
which by~\eqref{eqn:one} implies $\mu_j(k+1)\geq \mu_j(k).$
In view of $S_j(k+1)=S_j(k)$ and the inductive hypothesis, we further obtain
\[\mu_j(k)\geq \gamma^{|S_j(k)|-1}=\gamma^{|S_j(k+1)|-1},\]
implying $\mu_j(k+1)\ge \gamma^{|S_j(k+1)|-1}$. \\
\textit{Case $A_{S_j(k)\bar{S}_j(k)}(k+1)>0$:}
Since the chain is balanced, we have
\[A_{\bar{S}_j(k)S_j(k)}(k+1)\geq \al A_{S_j(k)\bar{S}_j(k)}(k+1)>0,\]
implying that
$A_{\bar{S}_j(k)S_j(k)}(k)>0$. Therefore, by the uniform boundedness
of $\{A(k)\}$, there exists
$\hat \xi\in \bar{S}_j(k)$ and $\hat \ell\in S_j(k)$ such that
$A_{\hat \xi \hat\ell}(k+1)\geq \gamma$. Hence, we have
\[A_{\hat \xi j}(k+1:0)\geq A_{\hat\xi \hat \ell}(k+1)A_{\hat \ell j}(k:0)
\geq \gamma \mu_j(k)= \gamma^{|S_j(k)|},\]
where the equality follows by the induction hypothesis.
Thus, $\hat \xi\in S_j(k+1)$ while $\hat\xi\not \in S_j(k)$, which implies
$|S_j(k+1)|\geq |S_j(k)|+1$. This, together with
$A_{\hat \xi j}(k+1:0)\geq \gamma^{|S_j(k)|}$, yields
$\mu_j(k+1)\geq \gamma^{|S_j(k)|}\geq \gamma^{|S_j(k+1)|-1}$.
\end{proof}

The bound on $\mu_j(k)$ of Lemma~\ref{lemma:infflowstable:boundedcolumns}
implies that the bound for the nonnegative entries given in
Proposition~4 of \cite{Lorenz} can be reduced from $\gamma^{m^2-m+2}$ to
$\gamma^{m-1}$.

Note that Lemma~\ref{lemma:infflowstable:boundedcolumns} holds for products
$A(k:t_0)$ starting with any $t_0\geq 0$,
(with appropriately defined $S_j(k)$ and $\mu_j(k)$). An immediate corollary
of Lemma~\ref{lemma:infflowstable:boundedcolumns} is the following result.
\begin{corollary}\label{cor:sumcolumn}
Under the assumptions of Lemma~\ref{lemma:infflowstable:boundedcolumns},
we have for all $k> t_0\ge0,$
\[\frac{1}{m}e^TA(k:t_0)\geq \min(\frac{1}{m},\gamma^{m-1})e^T,\]
where $e$ is the vector of ones and the inequality is
to be understood entry-wise.
\end{corollary}
\begin{proof}
Without loss of generality, let us assume that $t_0=0$. Then, by
Lemma~\ref{lemma:infflowstable:boundedcolumns} we have
$\frac{1}{m}e^TA^j(k:0)\geq \frac{1}{m}|S_j(k)|\gamma^{|S_j(k)|-1}$
for any $j\in[m]$, where $A^j$ denotes the $j$th column of $A$.
For $\gamma\in[0,1]$, the function
$t\mapsto t\gamma^{t-1}$ defined on $[1,m]$ attains its minimum at either
$t=1$ or $t=m$. Therefore,
$\frac{1}{m}e^TA(k:1)\geq \min(\frac{1}{m},\gamma^{m-1})e^T$.
\end{proof}

Now, we relax the assumption on the bounded entries in
Corollary~\ref{cor:sumcolumn}.
\begin{theorem}\label{thm:boundedcolumns}
Let $\Ac$ be a balanced and strongly aperiodic chain. Then, there is a scalar
$\gamma\in (0,1]$ such that
$\frac{1}{m}e^TA(k:0)\geq \min(\frac{1}{m},\gamma^{m-1})e^T$ for all
$k\geq 1$.
\end{theorem}
\begin{proof}
Let $\al>0$ be a balancedness coefficient for $\Ac$ and
let $A_{ii}(k)\geq \beta>0$ for all $i\in[m]$ and $k\geq 1$.
Further, let $\mathbf{B}_{\al,\beta}$ be the set of balanced matrices
with the balancedness coefficient $\al$ and strongly aperiodic matrices
with a coefficient $\beta>0$, i.e.,
\begin{align}\label{eqn:Bcal}
&\mathbf{B}_{\al,\beta}:=\left\{Q\in\mathbb{R}^{m\times m}\mid Q\geq 0,Qe=e,
\right.\\\nonumber
&\quad\left. Q_{S\bar{S}}\geq \al Q_{\bar{S}S}\mbox{ for all non-trivial
$S\subset [m]$},
Q_{ii}\geq \beta\mbox{ for all $i\in[m]$}\right\}.
\end{align}
The description in relation \eqref{eqn:Bcal} shows that
$\mathbf{B}_{\al,\beta}$ is
a bounded polyhedral set in $\R^{m\times m}$. Let
$\{ Q^{(\xi)}\in \mathbf{B}_{\al,\beta}\mid \xi\in [n_{\al,\beta}]\}$
be the set of extreme points of this polyhedral set indexed by the positive
integers between $1$ and $n_{\al,\beta},$ which is the total
number of extreme points of $\mathbf{B}_{\al,\beta}$.

Since $A(k)\in \mathbf{B}_{\al,\beta}$ for all $k\geq 1$, we can write
$A(k)$ as a convex combination of the extreme points in
$\mathbf{B}_{\al,\beta}$, i.e.,
there exist coefficients $\lambda_{\xi}(k)\in[0,1]$ such that
\begin{align}\label{eqn:initialeqn}
A(k)=\sum_{\xi=1}^{n_{\al,\beta}}\lambda_{\xi}(k)Q^{(\xi)}
\quad\hbox{with }\sum_{\xi=1}^{n_{\al,\beta}}\lambda_{\xi}(k)=1.
\end{align}
Now, consider the following independent random matrix process defined by:
\begin{align*}
W(k)=Q^{(\xi)}\quad \mbox{with probability $\lambda_{\xi}(k)$}
\qquad\hbox{for all }k\ge1.
\end{align*}
In view of this definition any sample path of $\Wc$
consists of extreme points of $\mathbf{B}_{\al,\beta}$. Thus,
every sample path of $\Wc$ has a coefficient bounded by
the minimum positive entry of the matrices in
$\{ Q^{(\xi)}\in \mathbf{B}_{\al,\beta}\mid \xi\in [n_{\al,\beta}]\}$,
denoted by $\gamma=\gamma(\al,\beta)>0$, where $\gamma>0$ since
$n_{\al,\beta}$ is finite.
Therefore, by Corollary~\ref{cor:sumcolumn}, we have
$\frac{1}{m}e^TW(k:t_0)\geq \min(\frac{1}{m},\gamma^{m-1})e^T$ for all
$k>t_0\geq 0$. Furthermore, by Eq.~\eqref{eqn:initialeqn} we have
$\EXP{W(k)}=A(k)$ for all $k\geq 1$, implying
\[\frac{1}{m}e^TA(k:t_0)=\frac{1}{m}e^T\EXP{W(k:t_0)}
\geq \ \min\left(\frac{1}{m},\gamma^{m-1}\right)e^T,\]
which follows from $\Wc$ being independent.
\end{proof}

Based on the above results, we are ready to prove the main result for
deterministic chains.
\begin{theorem}\label{thrm:balanceddeterministic}
Any balanced and strongly aperiodic chain $\Ac$ is in the class $\Pstar$.
\end{theorem}
\begin{proof}
  As pointed out in \cite{KolmogoroffMarkov} for any chain $\Ac$,
there exists a sequence $\{t_r\}$ of time indices, such that for all
$k\geq 0$, $\lim_{r\to\infty}A(t_r:k)=Q(k)$ exists and,
for any stochastic vector $\pi\in \Rm$, the sequence $\{Q^T(k)\pi\}$ is an
absolute probability sequence for $\Ac$.
Since $\Ac$ is a balanced and strongly aperiodic chain,
by Theorem~\ref{thm:boundedcolumns} it follows that
\[\frac{1}{m}e^TQ(k)=\frac{1}{m}\lim_{r\rightarrow\infty}e^TA(t_r:k)
\geq p^* e^T\qquad\hbox{for all }k\ge0,\]
with $p^*=\min(\frac{1}{m},\gamma^{m-1})>0$.
Thus, $\{\frac{1}{m}e^TQ(k)\}$ is a uniformly bounded
absolute probability sequence for $\Ac$.
\end{proof}

The main result of this section follows immediately from
Theorem~\ref{thrm:balanceddeterministic}.
\begin{theorem}\label{thrm:infflowstability:balancedadapted}
Any balanced and strongly aperiodic independent random chain
is in the class $\Pstar$.
\end{theorem}
\begin{proof}
The proof follows immediately by noticing that,
for an independent random chain, $\Wc$, any absolute probability sequence
for the expected chain $\{\EXP{W(k)}\}$ is an absolute probability process for
$\Wc$.
\end{proof}

As a result of Theorem~\ref{thrm:infflowstability:balancedadapted} and
Theorem~\ref{thrm:infflowstability:infiniteflowstability}, the proof of
Theorem~\ref{thrm:mainresult} follows immediately. In particular
by Theorem~\ref{thrm:infflowstability:balancedadapted}, any independent random
chain that is balanced and strongly aperiodic belongs to the class $\Pstar$.
Thus, the result follows by
Theorem~\ref{thrm:infflowstability:infiniteflowstability}.

\section{Connection to Non-negative Matrix Theory}\label{sec:connection}
In this section, we show that Theorem~\ref{thrm:mainresult} is a generalization
of the following well-known result in the non-negative matrix theory which
plays a central role in the theory of ergodic Markov chains.

\begin{lemma}\label{lemma:perron}(\cite{Kumar}, page 46)
For an aperiodic and irreducible stochastic matrix $A$, the limit
$\lim_{k\to\infty}A^k$ exists and it
is equal to a rank one stochastic matrix.
\end{lemma}

Recall that a stochastic matrix $A$ is irreducible if there is no
permutation matrix $P$ such that
    \[P^TAP=\left[
    \begin{array}{cc}
    X&Y\\\mathbf{0}&Z
  \end{array}\right],\]
where $X,Y,Z$ are $i\times i$, $i\times (m-i)$, and $(m-i)\times (m-i)$
matrices for some $i\in [m-1]$ and $\mathbf{0}$ is the $(m-i)\times i$ matrix
with all entries equal to zero.

Let us reformulate irreducibility using the tools we have developed in this
paper.
\begin{lemma}\label{lemma:implication:equivalency}
 A stochastic matrix $A$ is an irreducible matrix if and only if the static
chain $\{A\}$ is balanced and its infinite flow graph is connected.
 \end{lemma}
 \begin{proof}
By the definition, a matrix $A$ is irreducible if there is no permutation
matrix $P$ such that
    \[P^TAP=\left[
    \begin{array}{cc}
    X&Y\\\mathbf{0}&Z
  \end{array}\right].\]
Since $A$ is a non-negative matrix, we have that $A$ is reducible if and only
if there exists a subset $S=\{1,\ldots,i\}$ for some $i\in [m-1]$, such that
  \begin{align}\nonumber
0=[P^TAP]_{\bar{S}S}=\sum_{i\in \bar{S},j\in S}e_i[P^TAP]e_j
=\sum_{i\in \bar{S},j\in S}A_{\sigma_i\sigma_j}
=\sum_{i\in \bar{R},j\in {R}}A_{ij},
  \end{align}
where $\sigma_i=\{j\in[m]\mid Pe_i=e_j\}$ (which is a singleton since $P$ is
a permutation matrix) and $R=\{\sigma_i\mid i\in S\}$.
Thus, $A$ is irreducible if and only if $A_{S\bar{S}}>0$ for all non-trivial
$S\subset[m]$. Therefore, by letting
\[\al=\min_{\substack{S\subset[m]\\S\not=\emptyset}}
\frac{A_{S\bar{S}}}{A_{\bar{S}S}},\]
and noting that $\al>0$, we conclude that $\{A\}$ is balanced with
a balancedness coefficient $\al$. Furthermore,
since $A_{S\bar S}+A_{\bar S S}\ge A_{S\bar{S}}>0$ for all nontrivial
$S\subset[m]$,
it follows that the infinite flow graph of $\{A\}$ is connected.

Now,  suppose that $\{A\}$ is balanced and
its infinite flow graph of $\{A\}$ is connected.
Then, $A_{S\bar{S}}>0$ or $A_{\bar{S}S}>0$ for all non-trivial $S\subset [m]$.
By the balancedness of the chain it follows that
$\min(A_{S\bar{S}},A_{\bar{S}S})>0$ for any non-trivial $S\subset[m]$,
implying that $A$ is irreducible.
 \end{proof}

Note that for an aperiodic $A$, we can always find some $h\geq 1$ such that
$A^h_{ii}\geq \gamma>0$ for all $i\in[m]$. Thus, based on
Theorem~\ref{thrm:mainresult}, we have the following extension of
Lemma~\ref{lemma:perron} for independent random chains.

\begin{theorem}\label{thrm:implication:extensioncor}
Let $\Wc$ be a balanced and strongly aperiodic
independent random chain with a connected infinite flow graph. Then,
for any $t_0\geq 0$, the product $W(k:t_0)$ converges to a rank one stochastic
matrix almost surely (as $k$ goes to infinity). Moreover, if $\Wc$ does not have the infinite flow property, the product $W(k:t_0)$ almost surely converges to a (random) matrix that has rank at most $\tau$ for any $t_0\geq 0$, where $\tau$ is the number of connected components of the infinite flow graph of $\{\EXP{W(k)}\}$.
   \end{theorem}
   \begin{proof}
     The result follows immediately from Theorem~\ref{thrm:mainresult}.
   \end{proof}

An immediate consequence of Theorem~\ref{thrm:implication:extensioncor} is a
generalization of Lemma~\ref{lemma:perron} to inhomogeneous chains.
\begin{corollary}
  Let $\Ac$ be a balanced and strongly aperiodic stochastic chain. Then, $A(\infty:t_0)=\lim_{k\to\infty}A(k:t_0)$ exists for all $t_0\geq 0$. Moreover, $A(\infty:t_0)$ is a
rank one matrix for all $t_0\geq 0$ if and only if the infinite flow graph of $\Ac$ is connected.
\end{corollary}
\section{Conclusion}\label{sec:conclusion}
In this paper we studied the limiting behavior of the products of random stochastic matrices from the dynamic system point of view.
We showed that any dynamics driven by such products admits time-varying Lyapunov functions.
Then, we defined a class $\Pstar$ of random chains which possess a well-behaved limits.
We have introduced balanced chains and discussed how many of the previously well-studied random chains are examples of such chains. We have established a general stability result for product of random stochastic matrices and showed that this result extends a classical convergence result for time-homogeneous
irreducible and aperiodic Markov chains.
\bibliographystyle{amsplain}
\bibliography{annals-probability}

\providecommand{\bysame}{\leavevmode\hbox to3em{\hrulefill}\thinspace}
\providecommand{\MR}{\relax\ifhmode\unskip\space\fi MR }
\providecommand{\MRhref}[2]{%
  \href{http://www.ams.org/mathscinet-getitem?mr=#1}{#2}
}
\providecommand{\href}[2]{#2}
\begin{thebibliography}{10}

\bibitem{multiagent}
V.D. Blondel, J.M. Hendrickx, A.~Olshevsky, and J.N. Tsitsiklis,
  \emph{Convergence in multiagent coordination, consensus, and flocking},
  Proceedings of IEEE CDC, 2005.

\bibitem{bolouki2012theorems}
S.~Bolouki and R.P. Malhame, \emph{Theorems about ergodicity and
  class-ergodicity of chains with applications in known consensus models},
  Arxiv preprint arXiv:1204.6624 (2012).

\bibitem{DistCtrlRobotNetw}
F.~Bullo, J.~Cort\'es, and S.~Mart{\'\i}nez, \emph{Distributed control of
  robotic networks}, Applied Mathematics Series, Princeton University Press,
  2009, Electronically available at http://coordinationbook.info.

\bibitem{CaoMora}
M.~Cao, A.~S. Morse, and B.~D.~O. Anderson, \emph{Reaching a consensus in a
  dynamically changing environment: A graphical approach}, SIAM Journal on
  Control and Optimization \textbf{47} (2008), 575--600.

\bibitem{SenetaCons}
S.~Chatterjee and E.~Seneta, \emph{Towards consensus: Some convergence theorems
  on repeated averaging}, Journal of Applied Probability \textbf{14} (1977),
  no.~1, 89--97.

\bibitem{grenander1963probabilities}
U.~Grenander, \emph{Probabilities on algebraic structures}, vol.~6, Wiley
  Stockholm, 1963.

\bibitem{hajnal}
J.~Hajnal, \emph{The ergodic properties of non-homogeneous finite markov
  chains}, Proceedings of the Cambridge Philosophical Society \textbf{52}
  (1956), no.~1, 67--77.

\bibitem{Krause:02}
R.~Hegselmann and U.~Krause, \emph{Opinion dynamics and bounded confidence
  models, analysis, and simulation}, Journal of Artificial Societies and Social
  Simulation \textbf{5} (2002).

\bibitem{BalancedHenderickx}
J.M. Hendrickx and J.N. Tsitsiklis, \emph{Convergence of type-symmetric and
  cut-balanced consensus seeking systems},  (2011),
  \url{http://arxiv.org/abs/1102.2361}.

\bibitem{Jadbabaie03}
A.~Jadbabaie, J.~Lin, and S.~Morse, \emph{Coordination of groups of mobile
  autonomous agents using nearest neighbor rules}, IEEE Transactions on
  Automatic Control \textbf{48} (2003), no.~6, 988--1001.

\bibitem{Johansson08a}
B.~Johansson, \emph{On distributed optimization in networked systems}, Ph.D.
  thesis, Royal Institute of Technology, Stockholm, Sweden, 2008.

\bibitem{KolmogoroffMarkov}
A.~Kolmogoroff, \emph{Zur {Theorie} der {Markoffschen} {Ketten}}, Mathematische
  Annalen \textbf{112} (1936), no.~1, 155--160.

\bibitem{Krause:97}
U.~Krause, \emph{Soziale dynamiken mit vielen interakteuren. eine
  problemskizze.}, In Modellierung und Simulation von Dynamiken mit vielen
  interagierenden Akteuren (1997), 37–51.

\bibitem{Kumar}
P.R. Kumar and P.~Varaiya, \emph{Stochastic systems estimation, identification
  and adaptive control}, Information and System Sciences Series,
  Prentice--Hall, Englewood Cliffs New Jersey, 1986.

\bibitem{Lorenz}
J.~Lorenz, \emph{A stabilization theorem for continuous opinion dynamics},
  Physica A: Statistical Mechanics and its Applications \textbf{355} (2005),
  217–223.

\bibitem{martin2011continuous}
S.~Martin and A.~Girard, \emph{Continuous-time consensus under persistent
  connectivity and slow divergence of reciprocal interaction weights}, Arxiv
  preprint arXiv:1105.2755 (2011).

\bibitem{mukherjea1976measures}
A.~Mukherjea and N.A. Tserpes, \emph{Measures on topological semigroups},
  Springer, 1976.

\bibitem{alex_ange08}
A.~Nedi\'c, A.~Olshevsky, A.~Ozdaglar, and J.N. Tsitsiklis, \emph{Distributed
  subgradient methods and quantization effects}, Proceedings of the 47th IEEE
  Conference on Decision and Control, 2008.

\bibitem{Olshevsky09a}
\bysame, \emph{On distributed averaging algorithms and quantization effects},
  IEEE Transactions on Automatic Control \textbf{54} (2009), no.~11,
  2506--2517.

\bibitem{Nedic_cdc07}
A.~Nedi\'c and A.~Ozdaglar, \emph{On the rate of convergence of distributed
  subgradient methods for multi-agent optimization}, Proceedings of IEEE CDC,
  2007, pp.~4711--4716.

\bibitem{Nedic09}
\bysame, \emph{Distributed subgradient methods for multi-agent optimization},
  IEEE Transactions on Automatic Control \textbf{54} (2009), no.~1, 48--61.

\bibitem{Alex:Nonexistence}
A.~Olshevsky and J.N. Tsitsiklis, \emph{On the nonexistence of quadratic
  lyapunov functions for consensus algorithms}, IEEE Transactions on Automatic
  Control \textbf{53} (2008), no.~11, 2642--2645.

\bibitem{rosenblatt1960limits}
M.~Rosenblatt, \emph{Limits of convolution sequences of measures on a compact
  topological semigroup}, J. Math. Mech \textbf{9} (1960), 293--305.

\bibitem{rosenblatt1964products}
\bysame, \emph{Products of independent identically distributed stochastic
  matrices}, Tech. report, DTIC Document, 1964.

\bibitem{shi2011towards}
G.~Shi and K.H. Johansson, \emph{Towards a global agreement: the persistent
  graph}, Arxiv preprint arXiv:1112.1338 (2011).

\bibitem{Tahbaz-Salehi08}
A.~Tahbaz-Salehi and A.~Jadbabaie, \emph{A necessary and sufficient condition
  for consensus over random networks}, IEEE Transactions on Automatic Control
  \textbf{53} (2008), no.~3, 791--795.

\bibitem{Tahbaz-Salehi09}
\bysame, \emph{Consensus over ergodic stationary graph processes}, IEEE
  Transactions on Automatic Control \textbf{55} (2010), no.~1, 225--230.

\bibitem{TouriNedicCDC20112}
B.~Touri and A.~Nedi\'c, \emph{Alternative characterization of ergodicity for
  doubly stochastic chains}, to appear in IEEE Conference on Decision and
  Control 2011.

\bibitem{TouriNedicCDC2011}
\bysame, \emph{On existence of a quadratic comparison function for random
  weighted averaging dynamics and its implications}, to appear in IEEE
  Conference on Decision and Control 2011.

\bibitem{TouriNedich:Approx}
\bysame, \emph{On approximations and ergodicity classes in random chains},
  Under review, 2010.

\bibitem{TouriNedich:Doubly}
\bysame, \emph{On backward product of stochastic matrices}, Under review, 2011.

\bibitem{ErgodicityPaper}
B.~Touri and A.~Nedi\'{c}, \emph{On ergodicity, infinite flow and consensus in
  random models}, IEEE Transactions on Automatic Control \textbf{56} (2011),
  no.~7, 1593--1605.

\bibitem{Tsitsiklis84}
J.N. Tsitsiklis, \emph{Problems in decentralized decision making and
  computation}, Ph.D. thesis, Dept. of {E}lectrical {E}ngineering and
  {C}omputer {S}cience, {MIT}, 1984.

\bibitem{Tsitsiklis86}
J.N. Tsitsiklis, D.P. Bertsekas, and M.~Athans, \emph{Distributed asynchronous
  deterministic and stochastic gradient optimization algorithms}, IEEE
  Transactions on Automatic Control \textbf{31} (1986), no.~9, 803--812.

\bibitem{wolf}
J.~Wolfowitz, \emph{Products of indecomposable, aperiodic, stochastic
  matrices}, Proceedings of the American Mathematical Society \textbf{14}
  (1963), no.~4, 733--737.

\end{thebibliography}
\end{document}